\documentclass[11pt,reqno,oneside]{amsart}
\usepackage{amsmath, amsthm, amssymb}
\usepackage{tensor}
\usepackage{bbm}
\usepackage{geometry}
\usepackage{soul}
\usepackage[pdftex]{graphicx}
\usepackage{float}
\usepackage{subfig}
\usepackage{tikz} 
\usetikzlibrary{arrows,decorations.pathreplacing,decorations.markings,shapes.geometric}
\tikzset{
    bt/.style={draw=blue,thick},
    ns/.style={circle,draw=blue,fill=blue, inner sep=0pt, minimum size=2mm},
    string/.style={draw=#1, postaction={decorate}, decoration={markings,mark=at position .45 with {\arrow[blue]{triangle 60}}}},
    doublestring/.style={draw=#1, postaction={decorate}, decoration={markings, mark=at position .7 with {\arrow[blue]{triangle 60}}, 
    mark=at position .3 with {\arrowreversed[blue]{triangle 60}}}},
    costring/.style={draw=#1, postaction={decorate}, decoration={markings,mark=at position .55 with {\arrow[draw=#1]{<}}}},
    arr/.style={string=blue, thick},
    doublearr/.style={doublestring=blue, thick},
    lin/.style={blue},
    dlin/.style = {blue, dashed, thick},
    dot/.style={circle,draw=#1,fill=#1,inner sep=1pt},
}

\usepackage{dsfont}
\usepackage[all]{xy}
\graphicspath{{graphics/}}
\numberwithin{equation}{section}
\usepackage{fancyhdr}
\theoremstyle{definition}

\newtheorem{Thm}{Theorem}[section]

\newtheorem{Lemma}{Lemma}[section]

\newtheorem{Q}{Question}[section]

\newcommand{\sh}[1]{\mathcal{#1}}

\DeclareMathOperator{\SL}{SL}

\DeclareMathOperator{\C}{\mathbb{C}}

\DeclareMathOperator{\Z}{\mathbb{Z}}

\DeclareMathOperator{\diff}{d}

\DeclareMathOperator{\Lie}{Lie}

\DeclareMathOperator{\Proj}{\mathbb{P}}

\DeclareMathOperator{\Hom}{Hom}
\DeclareMathOperator{\length}{\ell}

\DeclareMathOperator{\Tangent}{T}

\DeclareMathOperator{\Set}{{\bf Set}}

\DeclareMathOperator{\Gm}{\mathbb{G}_m}
\DeclareMathOperator{\Ga}{\mathbb{G}_a}

\DeclareMathOperator{\Spec}{Spec}

\DeclareMathOperator{\Ho}{H}

\DeclareMathOperator{\ev}{ev}

\DeclareMathOperator{\Aff}{\mathbb{A}}

\DeclareMathOperator{\rot}{rot}

\DeclareMathOperator{\repart}{re}

\DeclareMathOperator{\Loop}{L}

\DeclareMathOperator{\fpqc}{fpqc}
\DeclareMathOperator{\etale}{\acute{e}t}
\DeclareMathOperator{\Alg}{{\bf Alg}}
\DeclareMathOperator{\Gp}{{\bf Gp}}
\DeclareMathOperator{\KM}{KM}
\DeclareMathOperator{\central}{cent}

\usepackage[pdftex,
bookmarks=true,
bookmarksnumbered=true,
hypertexnames=false,
colorlinks = true,
linkcolor = black,
citecolor = black,
urlcolor = black]{hyperref}

\title{Projective lines in the affine flag manifold with given tangent root vector}

\begin{document}
\author{Claude Eicher}
\date{\today}
\address{Department of Mathematics, Massachusetts Institute of Technology, Cambridge, MA 02139 United States}
\email{eicher@mit.edu}

\maketitle

\begin{abstract}
We first describe the tangent space to the affine flag manifold associated to a simple algebraic group over $\C$
at the distinguished point starting from standard definitions. We then construct projective lines in the affine
flag manifold tangent to given root vectors associated to imaginary roots of the corresponding affine Kac-Moody algebra
and describe in which Schubert varieties they lie. 
\end{abstract}

\tableofcontents

\section{Notation and Conventions}
An algebra over a field is understood to be commutative and unital. 
We denote by $\Alg/\C$, $\Set$, and $\Gp$ the category of $\C$-algebras, sets, and groups
respectively. $\Aff^d$ denotes the affine space of dimension $d$, $\Gm$ and $\Ga$ denotes the multiplicative and additive group. 

\section*{Acknowledgements}
The author would like to thank G. Felder for discussions on this topic and is supported by the
Swiss National Science Foundation. 

\section{Affine Kac-Moody Lie algebras}
Here we introduce some basic notation concerning
affine Kac-Moody Lie algebras. When $\mathfrak{g}$ is a simple $\C$-Lie algebra, we
denote by $\mathfrak{g}_{\KM}$ the corresponding untwisted affine Kac-Moody Lie algebra. 
The set of roots of $\mathfrak{g}_{\KM}$ is denoted by $\Phi_{\KM}$ and the subset of positive, negative, and real
roots by $\Phi^+_{\KM}$, $\Phi^-_{\KM}$, and $\Phi^{\repart}_{\KM}$ respectively. 
The \emph{indecomposable imaginary root} in $\Phi^+_{\KM}$ is denoted by $\delta$. 
We have triangular decompositions $\mathfrak{g} = \mathfrak{n}^- \oplus \mathfrak{h} \oplus \mathfrak{n}$
and $\mathfrak{g}_{\KM}Ê= \mathfrak{n}^-_{\KM}Ê\oplus \mathfrak{h}_{\KM}Ê\oplus \mathfrak{n}_{\KM}$.
Here, e.g. $\mathfrak{n}^-_{\KM}$ is the direct sum of the negative
root spaces $\mathfrak{n}^-_{\KM} = \bigoplus_{\alpha \in \Phi^-_{\KM}}\mathfrak{g}_{\KM \alpha}$. 
The reflections $s_{\alpha}$, $\alpha \in \Phi^{\repart +}_{\KM}$, generate the Weyl group $W$
of $\mathfrak{g}_{\KM}$. It is the semidirect product of the finite
Weyl group $W_{\mathfrak{g}}$ and the coroot lattice $\check{Q}$ of $\mathfrak{g}$, $W = W_{\mathfrak{g}}\ltimes  \check{Q}$. 
We will denote the Bruhat order on $W$ by $\leq$  and the length function
by $\length: W \rightarrow \Z_{\geq 0}$.  

\section{Affine flag manifold and Schubert varieties}
In \S  3.1-3.3 we recall the construction of the affine flag manifold and Schubert varieties following \cite{PR08}. 

\subsection{Affine flag manifold $X$}
Let $G$ be a simple linear algebraic group over $\C$
and $B \supseteq T$ a choice of a Borel and maximal torus subgroup.
Below we are going to assume the reader's familiarity with the basic concepts
of ind-schemes. An exposition of these can be found e.g. in \cite{Lev13}[Appendix A].
In contrast to this reference we are, however, going to consider an ind-scheme as a (covariant) functor
$\Alg/\C \rightarrow \Set$. 

The \emph{affine flag manifold} is
defined as in \cite{PR08}[equation (1.6)] as the quotient of fpqc-sheaves $\Loop G/ \Loop^+ I$,
where the Iwahori group $I$ is considered as a smooth affine group scheme over $\Spec \C[[t]]$. 
Below \emph{$X$ will always denote the affine flag manifold}. We recall
that $\Loop^+ I$ is (represented by) an affine group scheme over $\C$ not of finite type \cite{PR08}[section 1]
that is the preimage of $B$ under the natural morphism $\Loop^+ G \rightarrow G$. Here $\Loop^+ G$ 
denotes the similar construction where we consider the group scheme over $\Spec \C[[t]]$ induced by $G$.  Also,
$\Loop G$ is a group ind-affine ind-scheme over $\C$. 
We have a quotient morphism $\pi: \Loop G \rightarrow X$. 

The basic theorem about $X$ is
\begin{Thm} \cite{PR08}[Theorem 1.4]
$X$ is a strict ind-scheme of ind-finite type over $\C$. 
\end{Thm}
The unit $1 \in G(\C((t)))$ defines the \emph{distinguished point} of $X$.
We will also denote it by $1 \in X(\C)$.

\subsection{Weyl group $W$}\label{sec:Weylgroup}
\emph{For the remaining article we assume $G$ to be simply connected}. 
Let $N$ be the normalizer of $T$ in $G$. We have an isomorphism of groups
\begin{align*}
\check{Q} = \Hom(\Gm,T) \xrightarrow{\cong} T(\C((t)))/T(\C[[t]])\;,\; \gamma \mapsto \gamma(t)\;. 
\end{align*}
It extends to an isomorphism of the short exact sequences of groups
\begin{align*}
1 \rightarrow \check{Q} \rightarrow W \rightarrow W_{\mathfrak{g}} \rightarrow 1
\end{align*}
and
\begin{align*}
1 \rightarrow T(\C((t)))/T(\C[[t]]) \rightarrow N(\C((t)))/T(\C[[t]]) \rightarrow N(\C((t)))/T(\C((t))) \rightarrow 1\;.
\end{align*}
We define for $w \in W$ the point $\dot{w} \in X(\C)$
by any preimage of $w$ under the map $N(\C((t))) \rightarrow N(\C((t)))/T(\C[[t]])$. We recover the
distinguished point of $X$ as $\dot{1}$. 

\subsection{Schubert varieties $\overline{X_w}$}
We define the \emph{(affine) Schubert cell} $X_w$
and \emph{(affine) Schubert variety} $\overline{X_w}$ as in \cite{PR08}[Definition 8.3]
(denoted by $C_w$ and $S_w$ in loc. cit.). 
Then $X_w \cong \Aff^{\length(w)}$. $\overline{X_w}$ is a closed subscheme of $X$
and an irreducible projective variety carrying an action of $\Loop^+ I$.
We have closed embeddings $\overline{X_w} \hookrightarrow \overline{X_v}$ when
$w \leq v$. 

\begin{Thm} \cite{PR08}[Proposition 9.9]
$X = \varinjlim_{w \in W} \overline{X_w}$
\end{Thm}

\subsection{Kac-Moody group $G_{\KM}$}
In fact, $X$ can be constructed in terms of
the affine Kac-Moody group. We recall, cf. \cite{BD00}[section 7.15.1], that the affine Kac-Moody group
associated to $G$ can be described as $G_{\KM} = \Gm^{\rot} \ltimes \widehat{\Loop G}$, where the group of
loop rotations $\Gm^{\rot}=\Gm$ acts by scaling $t$. Here $\widehat{\Loop G}$ is 
a certain central extension of $\Loop G$ by $\Gm^{\central} = \Gm$
\begin{align*}
1 \rightarrow \Gm^{\central} \rightarrow \widehat{\Loop G} \xrightarrow{p} \Loop G \rightarrow 1 
\end{align*}
that splits canonically over $\Loop^+ G$. 
We define $\widehat{\Loop^+ I}=p^{-1}(\Loop^+ I)$ and 
$I_{\KM} = \Gm^{\rot} \ltimes \widehat{\Loop^+ I}$. Also we introduce 
$T_{\KM}$ as the quotient of $I_{\KM}$ by its pro-unipotent radical. 
Then $T_{\KM} \cong \Gm^{\rot} \times \Gm^{\central} \times T$ is an algebraic torus over $\C$. We have
a canonical identification $X
= G_{\KM}/I_{\KM}$ of fpqc-sheaves. The $X_w$ and $\overline{X_w}$
are $I_{\KM}$-invariant. It is known that the unique $T_{\KM}$-fixed point in $X_w$ is $\dot{w}$. 

\subsection{Points of $X$ with values in a strictly Henselian local ring}
\begin{Lemma}\label{Lemma:pointsofXthinwithvaluesinashring}
Let $R$ be a $\C$-algebra and a strictly Henselian local ring. The
map\\ $G(R((t)))/I(R[[t]]) \rightarrow X(R)$ defined by sheafification is a bijection. 
\end{Lemma}

\begin{proof}
Our reference for the fpqc topology is \cite{GW10}, \cite{Vis08}. 
Let $R$ be any $\C$-algebra. We have an ``exact" sequence of pointed sets
\begin{align*}
0 \rightarrow (\Loop^+ I)(R) \rightarrow (\Loop G)(R) \xrightarrow{\pi} X(R) \xrightarrow{d} \Ho^1((\Spec R)_{\fpqc},\Loop^+ I)\;. 
\end{align*}
Here $\Ho^1((\Spec R)_{\fpqc},\Loop^+I)$ denotes the pointed set of isomorphism classes of 
right torsors over the group scheme $\Loop^+ I$ on $\Spec R$ that become trivial on some fpqc cover of $\Spec R$.
The coboundary map $d$ sends $s \in X(R)$  to the class of the torsor $\sh{P}$ defined
by $\sh{P}(S)=\pi^{-1}(s \vert S)$ for $S \rightarrow \Spec R$ an object of the small fpqc site $(\Spec R)_{\fpqc}$. 
We have a similar set $\Ho^1((\Spec R)_{\etale},\Loop^+ I)$ for the \'etale topology
and an inclusion $\Ho^1((\Spec R)_{\etale},\Loop^+ I)\subseteq \Ho^1((\Spec R)_{\fpqc},\Loop^+ I)$. 
It was proven in \cite{HV11} that this inclusion is an equality (we can replace $\mathbb{F}_q$ by 
$\C$ in loc. cit.). If $R$ is a strictly Henselian local ring, then $\Ho^1((\Spec R)_{\etale},\Loop^+ I)$
is trivial. Indeed, this follows from the fact that any surjective \'etale morphism $S \rightarrow \Spec R$
has a section \cite{Mil80}[I, Theorem 4.2d)]. The statement of the lemma follows. 
\end{proof}

\section{Tangent space to $X$ at the distinguished point}
\subsection{Tangent space to an ind-scheme}\label{ssec:limitoftangentspaces}
Let $(Y_i)_{i \in I}$ be a directed family of schemes over $\C$
together with closed immersions $\iota_{ij}: Y_i \hookrightarrow Y_j$ for $i \leq j$. 
This defines a strict ind-scheme $Y=\varinjlim_{i \in I} Y_i$ over $\C$. Then 
$Y(R) = \varinjlim_{i \in I}Y_i(R)$ holds for any $\C$-algebra $R$. Let us 
furthermore assume that for each $i \in I$ we have a distinguished
point $1 \in Y_i(\C)$ such that $\iota_{ij}(1)=1$. The tangent space $\Tangent_1 Y_i$ is the preimage of $1$
under the map $Y_i(\C[\epsilon]/(\epsilon^2)) \rightarrow Y_i(\C)$ induced by the 
evaluation map $\C[\epsilon]/(\epsilon^2)  \rightarrow \C$ (this is an example of
a fiber product). In the same way we define $\Tangent_1 Y$. 
We have induced embeddings of $\C$-vector spaces $\Tangent_1 Y_i \hookrightarrow \Tangent_1 Y_j$ for $i \leq j$
and $\Tangent_1 Y = \varinjlim_{i \in I}\Tangent_1 Y_i$ holds since $\varinjlim_{i \in I}$
commutes with finite inverse limits.

\subsection{Lie algebra of a group-valued functor}
Let $\sh{G}: \Alg/\C \rightarrow \Gp$ 
be a (covariant) functor. We define $\Lie \sh{G}$ as $\Lie \sh{G} = \ker \ev$, 
where the group homomorphism $\ev: \sh{G}(\C[\epsilon]/(\epsilon^2)) \rightarrow \sh{G}(\C)$
is induced by the evaluation map $\C[\epsilon]/(\epsilon^2) \rightarrow \C$. 
If $\sh{G}$ is the functor of points of a $\C$-scheme $G$, $\Lie \sh{G}$ coincides with the tangent space
of $G$ at the identity, i.e. with the $\C$-vector space $(\mathfrak{m}/\mathfrak{m}^2)^*$, 
where $\mathfrak{m}$ is the maximal ideal corresponding to $1 \in G(\C)$, 
and is known to be a $\C$-Lie algebra. In particular, $\Lie \Loop^+ I$ is a 
$\C$-Lie algebra. More generally, according to \cite{BBE03}[2.10], 
$\Lie \sh{G}$ is naturally endowed with the structure of a $\C$-Lie algebra when $\sh{G}$ commutes with
finite inverse limits. As $\Loop G$ is a group ind-affine ind-scheme, its functor
of points indeed commutes with finite inverse limits, and hence $\Lie \Loop G$ is a $\C$-Lie algebra.  

\subsection{Tangent space}
In this section we return to the affine flag manifold $X$ and abbreviate $\C[\epsilon]/(\epsilon^2) = D$. 
We define the tangent space $\Tangent_1 X$ to $X$ at the distinguished point $1
\in X(\C)$ as in \autoref{ssec:limitoftangentspaces}. 

\begin{Lemma}\label{Lemma:tangentspacetoXthin}
We have a natural identification
$\Tangent_1 X = \mathfrak{n}^-_{\KM}=\bigoplus_{\alpha \in \Phi_{\KM}^-} \mathfrak{g}_{\KM \alpha}$
of $T_{\KM}$-modules. 
\end{Lemma}
\begin{proof}
We claim natural identifications
\begin{align*}
\ev^{-1}(1) = \left(\frac{G_{\KM}(D)}{
I_{\KM}(D)} \xrightarrow{\ev} \frac{G_{\KM}(\C)}{I_{\KM}(\C)}\right)^{-1}(1) = 
\frac{\ker \ev_{G_{\KM}}}{\ker \ev_{I_{\KM}}} = \frac{\Lie G_{\KM}}{\Lie I_{\KM}}
\end{align*}
of $T_{\KM}$-modules. The action of $T_{\KM}$ on $\Lie G_{\KM}$ is induced by
the adjoint action of $T_{\KM}$ on $G_{\KM}$ and similarly for $I_{\KM}$.
We have introduced the homomorphisms $\ev_{G_{\KM}}: G_{\KM}(D) \rightarrow G_{\KM}(\C)$ 
and $\ev_{I_{\KM}}: I_{\KM}(D) \rightarrow I_{\KM}(\C)$. In the first identification
we use Lemma \autoref{Lemma:pointsofXthinwithvaluesinashring} for the strictly Henselian local
rings $D$ and $\C$. Let us explain the second identification. 
There is clearly a natural map from the rhs to the lhs. Let
us construct a map from the lhs to the rhs. 
Consider the assignment $g \mapsto g\ev_{G_{\KM}} (g)^{-1}$. This defines
a map $G_{\KM}(D) \rightarrow G_{\KM}(D)$ by
considering $G_{\KM}(\C) \subseteq G_{\KM}(D)$. $\ev_{G_{\KM}}$ restricts to the identity
on this subgroup. Moreover $g \mapsto g \ev_{G_{\KM}}(g)^{-1}$ induces 
a map from the lhs to the rhs as 
$\ev_{G_{\KM}}(g \ev_{G_{\KM}}(g)^{-1})=1$ and for $h \in I_{\KM}(D)$ 
\begin{align*}
g h \ev_{G_{\KM}}(gh)^{-1} &= g h \ev_{G_{\KM}}(h)^{-1}
\ev_{G_{\KM}}(g)^{-1} \\
&= g \ev_{G_{\KM}}(g)^{-1}
\ev_{G_{\KM}}(g) h \ev_{G_{\KM}}(h)^{-1}\ev_{G_{\KM}}(g)^{-1} \equiv g \ev_{G_{\KM}}(g)^{-1} x\;.
\end{align*}
Note that $\ev_{G_{\KM}}(g) \in I_{\KM}(\C)$ and hence
$x \in I_{\KM}(D)$. Also
$\ev_{I_{\KM}}(x)=1$.
The two maps between the lhs and rhs are inverse to each other. Together 
with the description of $\Lie G_{\KM}$ following from the one of $\Lie \Loop G$ given in \cite{PR08}[Proposition 9.3] we conclude
the statement of the lemma. 
\end{proof}

The fact that $\overline{X_w}$ is invariant under the action of $T_{\KM}$ implies that $\Tangent_1 \overline{X_w}
\subseteq \Tangent_1 X$ is a $T_{\KM}$-submodule and hence has a basis of root vectors. 
In the remaining part of the article, we will study  the following two questions.

\begin{Q}\label{Q:rootvectorsintangentspace} Given $w \in W$, which root vectors lie in $\Tangent_1 \overline{X_w}$?
\end{Q}
\begin{Q}\label{Q:existenceofcurve}
Given $w \in W$ and a root vector $v$ in $\Tangent_1 X$, does there exist a projective line contained in $\overline{X_w}$
such that $v$ is tangent to it?
\end{Q}
Concerning the relation between these two questions, it is clear that an affirmative answer to the second question implies
 $v \in \Tangent_1 \overline{X_w}$, while the converse does not hold in general. 

\section{Projective lines associated to real roots}
In \cite{PR08}[section 9.h] it is shown that the $\overline{X_w}$ 
are isomorphic to the Schubert varieties of the Kac-Moody theory as defined in \cite{Ku02}[7.1.13]. 
In this section we review the description of the tangent space
to $\overline{X_w}$ given in \cite{Ku02}[Chapter XII] in order to provide a context for our result presented in \autoref{sec:projlinesassociatedtoimagroots} below. 
According to \cite{Ku02}[12.1.7 Proposition]
the one dimensional $T_{\KM}$-orbit closures in $\overline{X_w}$ containing $1$ are given by the closures
of the images of the root groups ($\cong \Ga$) associated to the elements
of the set 
\begin{align*}
\Phi_w = \{\alpha \in \Phi_{\KM}^{\repart -}\; \vert\; s_{-\alpha} \leq w\}\;. 
\end{align*} 
The orbit closure associated to $\alpha$ is a projective line
contained in $\overline{X_w}$ and containing the $T_{\KM}$-fixed
points $1$ and $\dot{s_{-\alpha}}$. These orbit closures are all distinct. Thus, the answer to Question \autoref{Q:existenceofcurve} is affirmative
in case $v$ is a root vector associated to a root from $\Phi_w$.  
It immediately follows \cite{Ku02}[12.1.10 Corollary] 
\begin{align}\label{eq:realrootscontained}
\Tangent_1 \overline{X_w} \supseteq \bigoplus_{\alpha \in \Phi_w}\mathfrak{g}_{\KM \alpha}\;.
\end{align}
This provides a partial answer to Question \autoref{Q:rootvectorsintangentspace}
in the sense that it describes root vectors contained in $\Tangent_1 \overline{X_w}$.
In general, we have
$\vert \Phi_w\vert \geq \length(w)$ \cite{Ku02}[12.1.8 Corollary]. Since the singular locus of $\overline{X_w}$
is a union of Schubert subvarieties of $\overline{X_w}$, $1$ is a smooth point of $\overline{X_w}$ if and only if
$\overline{X_w}$ is smooth. Moreover, $\dim \overline{X_w} = \length(w)$ implies that
if $\overline{X_w}$ is smooth, then \eqref{eq:realrootscontained} is an equality and $\vert \Phi_w\vert = \length(w)$.
In particular, in this case $\Tangent_1 \overline{X_w}$ has a basis of root vectors associated to \emph{real} roots. 
 In \autoref{ssec:limitoftangentspaces} it was explained that $\Tangent_1 X
= \varinjlim_{w \in W} \Tangent_1 \overline{X_w}$. In conjunction with Lemma \autoref{Lemma:tangentspacetoXthin}
this shows that not all $\overline{X_w}$ can be smooth as $\Phi^-_{\KM}$ contains imaginary roots.
If $1$ is only a \emph{rationally smooth} point of $\overline{X_w}$, see  \cite{Ku02}[12.2.7 Definition], then \eqref{eq:realrootscontained}
is not in general an equality, but it is still known that $\vert \Phi_w\vert = \length(w)$ \cite{Ku02}[12.2.14 Theorem]. 

\section{Projective lines associated to imaginary roots}\label{sec:projlinesassociatedtoimagroots}
The following theorem is the main result of the article. Its proof constructs a projective line 
in the affine flag manifold of $\SL_2$ that is tangent, at the distinguished point, to a root vector associated to an imaginary root.  
The statement of the theorem and its proof should be compared with \cite{Hei10}[Lemma 6].

\begin{Thm}\label{Prop:SL2P1imagroot}
Let $G=\SL_2$ and $n \in \Z_{>0}$. Let $\check{\alpha} \in \check{Q}$
be the positive coroot. Let $v \in \Tangent_1 X$ be a root vector associated to the imaginary root
$-n\delta$ ($v$ is unique up to nonzero scalar). There is a one dimensional $\Gm^{\rot}$-orbit $C$ in $\overline{X_{-3n\check{\alpha}}}$ such
that $\overline{C}\setminus C =  1 \sqcup \dot{-3n\check{\alpha}}$ and $v$ is tangent to $\overline{C}$. 
$\overline{C}$ is isomorphic to a projective line. 
\end{Thm}
\begin{proof}
\emph{1. Construction of a morphism $f: \Aff^1 \rightarrow \overline{X_w}$.}\\
Let $c=1+\epsilon t^{-n}$. We have the identity in $\SL_2((\C[\epsilon]/(\epsilon^2))((t)))$
\begin{align}\label{eq:expressionfortoruselement}
\begin{pmatrix} c & 0 \\ 0 & c^{-1}\end{pmatrix} = \begin{pmatrix} 1 & c\\ 0 & 1\end{pmatrix}
\begin{pmatrix} 1 & 0 \\ d & 1\end{pmatrix} \begin{pmatrix} 1 & c\\ 0 & 1\end{pmatrix} 
\begin{pmatrix} 0 & -1\\ 1 & 0\end{pmatrix}
\end{align}
with $cd=-1$. I.e. $d=-1+\epsilon t^{-n}$. We can consider the image
of the lhs in $\Tangent_1 X$. Then any root vector associated to the root $-n\delta$ is a nonzero
multiple of it. We claim that if we consider $c,d$ instead as
elements of $\C[\epsilon]((t))$, the rhs of this equation defines a morphism $f: \Aff^1 \rightarrow
\overline{X_w}$, i.e. an element of $\overline{X_w}(\C[\epsilon])$, such that $f(0)=1$
and the differential $\diff f(0)$ at $0$ is a root vector associated to the root $-n \delta$. 
Indeed, the entries of these matrices lie in $\C[\epsilon]((t))$ and hence
the product of the matrices defines points in $\overline{X_w}$ for some $w \in W$ depending on $n$. Furthermore it is clear that
the image of $f$ is the point $1$ together with a one dimensional $\Gm^{\rot}$-orbit $C$, see also 3. below.\\

\emph{2. We have $\overline{f}(\infty) = \begin{pmatrix} t^{-3n} & 0 \\ 0 & t^{3n}\end{pmatrix}1
= \dot{-3n\check{\alpha}}$ and $\overline{C} \subseteq \overline{X_{-3n\check{\alpha}}}$.}\\
Here $\overline{f}: \Proj^1 \rightarrow \overline{X_w}$ is the unique extension of $f$, which exists
because $\overline{X_w}$ is projective. The rhs of \eqref{eq:expressionfortoruselement} 
considered as an element of $\SL_2((t))(\C[\epsilon])$ is
\begin{align*}
\begin{pmatrix} 1 + \epsilon t^{-n} + \epsilon^2 t^{-2n}+\epsilon^3 t^{-3n} & -\epsilon^2 t^{-2n} \\
\epsilon^2 t^{-2n} & 1-\epsilon t^{-n}\end{pmatrix}\;. 
\end{align*}
It is easy to show that
\begin{align}\label{eq:expressionforcd}
\begin{split}
& \begin{pmatrix} 1 + \epsilon t^{-n} + \epsilon^2 t^{-2n}+\epsilon^3 t^{-3n} & -\epsilon^2 t^{-2n} \\
\epsilon^2 t^{-2n} & 1-\epsilon t^{-n}\end{pmatrix} 1
\\
& = \begin{pmatrix} \epsilon^{-2}+\epsilon^{-1}t^{-n}+t^{-2n}+\epsilon t^{-3n} & t^{2n} \\
-t^{-2n} & 0 \end{pmatrix}1\\
& = \begin{pmatrix} \epsilon^{-3}+\epsilon^{-2}t^{-n}+\epsilon^{-1}t^{-2n}+t^{-3n} & -\epsilon^{-2}t^{5n}-\epsilon^{-1}t^{4n}
-t^{3n} \\ -\epsilon^{-1}t^{-2n} & t^{3n}\end{pmatrix}1
\end{split}
\end{align}
holds in $X(\C[\epsilon,\epsilon^{-1}])$
by multiplying with suitable elements in $(\Loop^+ I)(\C[\epsilon,\epsilon^{-1}])$ from the right. From the second line in  \eqref{eq:expressionforcd} we conclude
 $\overline{C} \subseteq \overline{X_w}$ for $w=-3n\check{\alpha}$ using the
explicit description of the Schubert cells in the case of $\SL_2$ (we will not recall this well-known description here). From the third line in \eqref{eq:expressionforcd} we see
\begin{align*}
\overline{f}(\infty) = \begin{pmatrix} t^{-3n} & -t^{3n} \\ 0 & t^{3n}\end{pmatrix}1 =\begin{pmatrix} t^{-3n} & 0 \\ 0 & t^{3n}\end{pmatrix}1\;. 
\end{align*}

\emph{3. $\overline{C}$ is isomorphic to a projective line.}\\
The action of $\Gm^{\rot}$ on $X$ defines an action morphism $ \Gm^{\rot} \rightarrow C$,
$\lambda \mapsto \lambda \cdot f(\epsilon)$, for fixed $\epsilon \neq 0$.
In fact we have $\lambda \cdot f(\epsilon)  = f(\lambda^{-n}\epsilon)$. While this action morphism might not be an isomorphism, it yields 
an isomorphism $\alpha: \Gm \xrightarrow{\cong} C$, which extends
to a morphism $\overline{\alpha}: \Proj^1 \rightarrow \overline{C}$. 
By 2. $\overline{\alpha}$ is bijective. The differential of $\overline{f}$
satisfies $\diff f(0) \neq 0$ and by inspecting the $\epsilon^{-1}$-term in the
third line of \eqref{eq:expressionforcd} we see $\diff \overline{f}(\infty) \neq 0$. Indeed, 
the automorphism of $X$ defined by $\begin{pmatrix} t^{3n} & 0 \\ 0 & t^{-3n}\end{pmatrix}$
induces an isomorphism $\Tangent_{\dot{-3n\check{\alpha}}}X \xrightarrow{\cong} \Tangent_1 X$. 
Now one checks that the image of $\diff \overline{f}(\infty)$ under this isomorphism
is given by  
\begin{align*}
\left[Ê\begin{pmatrix} t^n & 0 \\ -t^{-5n} & -t^n\end{pmatrix}\right] = \left[\begin{pmatrix} 0 & 0 \\ -t^{-5n} & 0\end{pmatrix}\right] \in \Tangent_1 X\;,   
\end{align*}
which is nonzero. This implies
that $\diff \overline{\alpha}$ is everywhere nonzero. By \cite{Har92}[Corollary 14.10] $\overline{\alpha}$ is an isomorphism. 
\end{proof}

For general $G$, we deduce the following theorem. It implies
an affirmative answer to Question \autoref{Q:existenceofcurve} when $w=-3n\check{\alpha}$ and
$v$ is the coroot $\check{\alpha}$ considered as a root vector associated to the imaginary root $-n\delta$. 
\begin{Thm}
Let $n \in \Z_{>0}$. Let $\check{\alpha} \in \mathfrak{h}Ê\cong \mathfrak{g}_{\KM-n \delta}$ be a positive coroot of $G$. 
Then there is $\Gm^{\rot}$-orbit $C$ in $\overline{X_{-3n\check{\alpha}}}$
such that $\overline{C}\setminus C = 1 \sqcup \dot{-3n\check{\alpha}}$
and $\check{\alpha}$ is tangent to $\overline{C}$. $\overline{C}$ is
isomorphic to a projective line. 
\end{Thm}

\begin{proof}
Let $\alpha$ be a positive root of $\mathfrak{g}$. We have an
embedding of Lie algebras $\mathfrak{sl}_2 \hookrightarrow \mathfrak{g}$
that sends $e \mapsto e_{\alpha}$ and $f \mapsto f_{\alpha}$,
where $e_{\alpha}$ and $f_{\alpha}$ is the standard generator in $\mathfrak{g}_{\alpha}$ and $\mathfrak{g}_{-\alpha}$ respectively.
This map is the differential at the identity of a unique closed embedding $\iota: \SL_2 \hookrightarrow G$
of algebraic groups. 
Let us indicate that an object is constructed for $\SL_2$ instead of $G$ by a superscript $^{(\SL_2)}$. 
The map $\iota$ induces an injective group homomorphism $\omega: W^{(\SL_2)} \hookrightarrow W$,
by the definition of the Weyl group given in \autoref{sec:Weylgroup}. $\omega$
embeds the coroot lattice of $\SL_2$ into the one for $G$ by sending the positive coroot of $\SL_2$ to $\check{\alpha}$.
Moreover $\omega$ embeds the finite Weyl group of $\SL_2$ 
into the one for $G$ by sending the reflection to $s_{\alpha} = e^{e_{\alpha}}e^{-f_{\alpha}}e^{e_{\alpha}} \in N/T$.
We claim that $\iota$ induces a closed embedding 
of Schubert varieties $\overline{X_w}^{(\SL_2)} \hookrightarrow \overline{X_{\omega(w)}}$
for each $w \in W^{(\SL_2)}$. To prove this, we first note that $\iota$
induces a closed embedding of ind-schemes $\phi: X^{(\SL_2)} \hookrightarrow X$.
This follows from \cite{BD00}[lemma in the proof of Theorem 4.5.1] and the fact that the quotient
$G/\iota(\SL_2)$ is affine. Let $S$ be a quasi-compact locally closed subscheme
of $X^{(\SL_2)}$. Then $S$ is a locally closed subscheme of 
$\overline{X_w}^{(\SL_2)}$ for some $w \in W^{(\SL_2)}$.
Furthermore $\phi$ induces an isomorphism between the reduced subscheme 
structure on the Zariski closure of $S$ in $X^{(\SL_2)}$ and the reduced subscheme
structure on the Zariski closure of $\phi(S)$ in $X$. 
We take $S=X_w^{(\SL_2)}$, a Schubert cell, and note that
the restriction of $\phi$ to $S$
defines a closed embedding $S \rightarrow X_{\omega(w)}$. Hence 
the reduced subscheme defined on the closure of $\phi(S)$ in $X$
is a closed subvariety of the Schubert variety $\overline{X_{\omega(w)}}$. It follows 
that $\phi$ induces a closed embedding $\overline{X_w}^{(\SL_2)} \hookrightarrow \overline{X_{\omega(w)}}$.
Theorem \autoref{Prop:SL2P1imagroot} now implies the statement. 
\end{proof}

\bibliographystyle{alpha}
\bibliography{references}

\begin{thebibliography}{Kum02}

\bibitem[BBE02]{BBE03}
A.~A. Beilinson, S.~J. Bloch, and H.~Esnault.
\newblock $\epsilon$-factors for {Gauss-Manin} determinants.
\newblock {\em Mosc. Math. J.}, 2(3):477--532, 2002.

\bibitem[BD]{BD00}
A.~Beilinson and V.~Drinfeld.
\newblock {\em Quantization of {Hitchin}'s integrable system and {Hecke}
  eigensheaves}.
\newblock unpublished.

\bibitem[GW10]{GW10}
U.~Goertz and T.~Wedhorn.
\newblock {\em Algebraic Geometry I}.
\newblock Vieweg+Teubner, 2010.

\bibitem[Har92]{Har92}
J.~Harris.
\newblock {\em Algebraic Geometry. A First Course}.
\newblock Springer, 1992.

\bibitem[Hei10]{Hei10}
J.~Heinloth.
\newblock Uniformization of $\mathcal{G}$-bundles.
\newblock {\em Math. Ann.}, 347:499--528, 2010.

\bibitem[HV11]{HV11}
U.~Hartl and E.~Viehmann.
\newblock The {Newton} stratification on deformations of local {$G$}-shtukas.
\newblock {\em J. reine angew. Math.}, 656:87--129, 2011.

\bibitem[Kum02]{Ku02}
S.~Kumar.
\newblock {\em {Kac-Moody} groups, their flag varieties and representation
  theory}.
\newblock {Birkh\"auser}, 2002.

\bibitem[Lev13]{Lev13}
B.~W.~A. Levin.
\newblock {\em G-valued flat deformations and local models}.
\newblock PhD thesis, Stanford University, 2013.

\bibitem[Mil80]{Mil80}
J.~S. Milne.
\newblock {\em Etale cohomology}.
\newblock Princeton University Press, 1980.

\bibitem[PR08]{PR08}
G.~Pappas and M.~Rapoport.
\newblock Twisted loop groups and their affine flag varieties.
\newblock {\em Advances in Mathematics}, 219:118--198, 2008.

\bibitem[Vis05]{Vis08}
A.~Vistoli.
\newblock Notes on {Grothendieck} topologies, fibered categories and descent
  theory.
\newblock In {\em Fundamental Algebraic Geometry. Grothendieck's FGA
  explained}, Mathematical Surveys and Monographs. A. M. S., 2005.

\end{thebibliography}

\end{document}